\documentclass[11pt,reqno]{amsart}

\usepackage[off]{auto-pst-pdf}

\usepackage[square]{natbib} 

\usepackage{fullpage,times,graphicx,amssymb,amsmath,amsfonts,bbm,psfrag,xcolor}
\usepackage[colorlinks,citecolor=bbluegray,linkcolor=ddarkbrown,urlcolor=blue,breaklinks]{hyperref}

\newtheorem{theorem}{Theorem}[section]
\newtheorem{proposition}[theorem]{Proposition}

\newtheorem{corollary}[theorem]{Corollary}

\usepackage{algorithmic,algorithm}

\usepackage[square]{natbib}

\definecolor{ddarkbrown}{rgb}{0.5,0.2,0.05} \definecolor{bbluegray}{rgb}{0.05,0,0.5}

\newcommand{\BEAS}{\begin{eqnarray*}}
\newcommand{\EEAS}{\end{eqnarray*}}
\newcommand{\BEA}{\begin{eqnarray}}
\newcommand{\EEA}{\end{eqnarray}}
\newcommand{\BEQ}{\begin{equation}}
\newcommand{\EEQ}{\end{equation}}
\newcommand{\BIT}{\begin{itemize}}
\newcommand{\EIT}{\end{itemize}}
\newcommand{\BNUM}{\begin{enumerate}}
\newcommand{\ENUM}{\end{enumerate}}

\newcommand{\BA}{\begin{array}}
\newcommand{\EA}{\end{array}}

\newcommand{\ones}{\mathbf 1}

\newcommand{\reals}{{\mathbb R}}

\newcommand{\E}{\mathbb{E}}

\newcommand{\idm}{\mathbf{I}}

\newcommand{\argmin}{\mathop{\rm argmin}}

\title{Nonlinear Acceleration of Stochastic Algorithms}

\author{Damien Scieur}
\address{INRIA \& D.I.,\vskip 0ex
\'Ecole Normale Sup\'erieure, Paris, France.}
\email{damien.scieur@inria.fr}

\author{Alexandre d'Aspremont}
\address{CNRS \& D.I., UMR 8548,\vskip 0ex
\'Ecole Normale Sup\'erieure, Paris, France.}
\email{aspremon@ens.fr}

\author{Francis Bach}
\address{INRIA \& D.I.\vskip 0ex
\'Ecole Normale Sup\'erieure, Paris, France.}
\email{francis.bach@inria.fr}


\usepackage[section]{placeins}

\begin{document}

\begin{abstract}
Extrapolation methods use the last few iterates of an optimization algorithm to produce a better estimate of the optimum. They were shown to achieve optimal convergence rates in a deterministic setting using simple gradient iterates. Here, we study extrapolation methods in a stochastic setting, where the iterates are produced by either a simple or an accelerated stochastic gradient algorithm. We first derive convergence bounds for arbitrary, potentially biased perturbations, then produce asymptotic bounds using the ratio between the variance of the noise and the accuracy of the current point. Finally, we apply this acceleration technique to stochastic algorithms such as SGD, SAGA, SVRG and Katyusha in different settings, and show significant performance gains.
\end{abstract}

\maketitle

\section{Introduction}

We focus on the problem
\BEA
\min_{x\in \mathbb{R}^d} f(x) \label{eq:min_problem_intro}
\EEA
where $f(x)$ is a smooth and strongly convex function with respect to the Euclidean norm. We consider a stochastic first-order oracle, which gives a noisy estimate of the gradient of $f(x)$, with
\BEA
\nabla_\varepsilon f(x) = \nabla f(x) + \varepsilon, \label{eq:stoch_oracle}
\EEA
where $\varepsilon$ is a noise term with bounded variance. This is the case for example when $f$ is a sum of strongly convex functions, and we only have access to the gradient of one randomly selected function. Stochastic optimization \eqref{eq:stoch_oracle} is typically challenging as classical algorithms are not convergent (for example, gradient descent or Nesterov's accelerated gradient). Even the averaged version of stochastic gradient descent with constant step size does not converge to the solution of \eqref{eq:min_problem_intro}, but to another point whose proximity to the real minimizer depends of the step size \citep{nedic2001convergence,moulines2011non}.

When $f$ is a finite sum of $N$ functions, then algorithms such as SAG \citep{schmidt2013minimizing}, SAGA \citep{defazio2014saga}, SDCA \citep{shalev2013stochastic} and SVRG \citep{johnson2013accelerating} accelerate convergence using a variance reduction technique akin to control variate in Monte-Carlo methods. Their rate of convergence depends of $1-\mu/L$ and thus does exhibit an accelerated rate on par with the deterministic setting (in $1-\sqrt{\mu/L}$). Recently a generic acceleration algorithm called Catalyst \citep{lin2015universal}, based on the proximal point methods improved this rate of convergence, at least in theory. Unfortunately, numerical experiments show this algorithm to be conservative, thus limiting practical performances. On the other hand, recent papers, for example \citep{shalev2014accelerated} (Accelerated SDCA) and \citep{allen2016katyusha} (Katyusha), propose algorithms with accelerated convergence rates, if the strong convexity parameter is given.

When $f$ is a quadratic function then averaged SGD converges, but the rate of decay of initial conditions is very slow. Recently, some results have focused on accelerated versions of SGD for quadratic optimization, showing that with a two step recursion it is possible to enjoy both the optimal rate for the bias term and the variance \citep{flammarion2015averaging}, given an estimate of the ratio between the distance to the solution and the variance of $\varepsilon$.

A novel generic acceleration technique was recently proposed by \citet{scieur2016regularized} in the deterministic setting. This uses iterates from a slow algorithm to extrapolate estimates of the solution with asymptotically optimal convergence rate. Moreover, this rate is reached \textit{without prior knowledge of the strong convexity constant}, whose online estimation is still a challenge, even in the deterministic case \citep{fercoq2016restarting}.

Convergence bounds are derived by \citet{scieur2016regularized}, tracking the difference between the deterministic first-order oracle of \eqref{eq:min_problem_intro} and iterates from a linearized model. The main contribution of this paper is to extend the analysis to arbitrary perturbations, including stochastic ones, and to present numerical results when this acceleration method is used to speed up stochastic optimization algorithms.

In Section~2 we recall the extrapolation algorithm, and quickly summarize its main convergence bounds in Section~3. In Section~4, we consider a stochastic oracle and analyze its asymptotic convergence in Section 5. Finally, in Section 6 we describe numerical experiments which confirm the theoretical bounds and show the practical efficiency of this acceleration.

\section{Regularized Nonlinear Acceleration}

Consider the optimization problem
\[
	\min_{x\in \mathbb{R}^d}f(x)
\]
where $f(x)$ is a $L-$smooth and $\mu-$strongly convex function \citep{Nest03a}. Applying the fixed-step gradient method to this problem yields the following iterates
\BEQ
	\tilde x_{t+1} = \tilde x_t - \frac{1}{L} \nabla f(\tilde x_t). \label{eq:grad_method}
\EEQ
Let $x^*$ be the unique optimal point, this algorithm is proved to converge with
\BEQ
	\|\tilde x_t-x^*\| \leq (1-\kappa)^t\|\tilde x_0-x^*\| \label{eq:rate_gradient}
\EEQ
where $\|\cdot\|$ stands for the $\ell_2$ norm and $\kappa = {\mu}/{L} \in [0,1[$ is the (inverse of the) condition number of $f$ \citep{Nest03a}. Using a two-step recurrence, the \emph{accelerated gradient descent} by \citet{Nest03a} achieves an improved convergence rate
\BEQ
	\|\tilde x_t-x^*\| \leq O\Big((1-\sqrt{\kappa})^t\|\tilde x_0-x^*\|\Big). \label{eq:rate_nesterov}
\EEQ
Indeed, \eqref{eq:rate_nesterov} converges faster than \eqref{eq:rate_gradient} but the accelerated algorithm requires the knowledge of $\mu$ and $L$. Extrapolation techniques however obtain a similar convergence rate, but do not need estimates of the parameters $\mu$ and $L$. The idea is based on the comparison between the process followed by $\tilde x_i$ with a \textit{linearized} model around the optimum, written
\[
	x_{t+1} = x_{t}- \frac{1}{L}\left( \nabla f(x^*) + \nabla^2f(x^*)(x_t-x^*)\right), \quad x_0 = \tilde x_0. 
\]
which can be rewritten as
\BEQ\label{eq:linear_algo}
x_{t+1}-x^*=\Big(\idm-\frac{1}{L}\nabla^2f(x^*)\Big) (x_t-x^*), \quad x_0 = \tilde x_0.
\EEQ
A better estimate of the optimum in \eqref{eq:linear_algo} can be obtained by forming a linear combination of the iterates (see \citep{anderson1965iterative,Caba76,Mesi77}), with
\[
	\Big\|\sum_{i=0}^t c_i x_i -x^*\Big\| \ll \|x_t-x^*\|,
\]
for some specific $c_i$ (either data driven, or derived from Chebyshev polynomials). These procedures were limited to quadratic functions only, i.e. when $\tilde x_i = x_i$ but this was recently extended to generic convex problems by \citet{scieur2016regularized} and we briefly recall these results below.

To simplify the notations, we define the function $g(x)$ and the step
\BEQ
	\tilde x_{t+1} = g(\tilde x_t), \label{eq:fixed_point_iteration}
\EEQ
where $g(x)$ is differentiable, Lipchitz-continuous with constant $(1-\kappa) <1$, $g(x^*) = x^*$ and $g'(x^*)$ is symmetric. For example, the gradient method \eqref{eq:grad_method} matches exactly this definition with $g(x) = x-\nabla f(x)/L$. Running $k$ steps of \eqref{eq:fixed_point_iteration} produces a sequence $\{\tilde x_0,\,...,\, \tilde x_k\}$, which we extrapolate using Algorithm \ref{algo:acc_fixedpoint_reg} from \citet{scieur2016regularized}.

\begin{algorithm}[H]
	\caption{Regularized Nonlinear Acceleration (\textbf{RNA})}
	\label{algo:acc_fixedpoint_reg}
	\begin{algorithmic}[1]
		\REQUIRE Iterates $\tilde x_0, \tilde x_1,..., \tilde x_{k+1}\in\reals^d$ produced by~\eqref{eq:fixed_point_iteration}, and a regularization parameter $\lambda > 0$.
		\STATE Compute $\tilde R = [\tilde r_0,...,\tilde r_k]$, where $\tilde r_i = \tilde x_{i+1}-\tilde x_i$ is the $i^{th}$ residue.
		\STATE Solve 
		\[
		\tilde c^\lambda = \argmin\limits_{c^T1 = 1} \|\tilde Rc\|^2 + \lambda \|c\|^2,
		\] 
		or equivalently solve $(\tilde R^T\tilde R+\lambda I) z = \ones$ then set $\tilde c^\lambda = {z}/{\ones^Tz}$.
		\ENSURE Approximation of $x^*$ computed as $\sum_{i=0}^{k} \tilde c^{\lambda}_i \tilde x_i$
	\end{algorithmic}
\end{algorithm}

For a good choice of $\lambda$, the output of Algorithm \eqref{algo:acc_fixedpoint_reg} is a much better estimate of the optimum than $\tilde x_{k+1}$ (or any other points of the sequence). Using a simple grid search on a few values of $\lambda$ is usually sufficient to improve convergence (see \citep{scieur2016regularized} for more details).

\section{Convergence of Regularized Nonlinear Acceleration}

We quickly summarize the argument behind the convergence of Algorithm \eqref{algo:acc_fixedpoint_reg}. The theoretical bound compare $\tilde x_i$ to the iterates produced by the linearized model
\BEQ
	x_{t+1} = x^* + \nabla g(x^*)(x_t-x^*), \quad x_0 = \tilde x_0. \label{linearized_fixed_point}
\EEQ

We write $c^\lambda$  the coefficients computed by Algorithm \eqref{algo:acc_fixedpoint_reg} from the ``linearized'' sequence $\{x_0,\,...,\, x_{k+1}\}$ and the error term can be decomposed into three parts,
\BEQ
	\Big\| \sum_{i=0}^{k} \tilde c^{\lambda}_i \tilde x_i - x^*\Big\| \leq 
	\underbrace{ \Big\|\sum_{i=0}^{k}  c^{\lambda}_i  x_i - x^* \Big\| }_{\textbf{Acceleration}}
	+ \underbrace{ \Big\|\sum_{i=0}^{k} \Big( \tilde c^{\lambda}_i - c^{\lambda}_i \Big) (x_i-x^*) \Big\| }_{\textbf{Stability}}
	+ \underbrace{ \Big\|\sum_{i=0}^{k} \tilde c^{\lambda}_i \Big(\tilde x_i-x_i\Big) \Big\|}_{\textbf{Nonlinearity}}.  \label{eq:bound_acc_split}
\EEQ
Convergence is guaranteed as long as the errors $(\tilde x_i-x^*)$ and $(x_i-\tilde x_i)$ converge to zero fast enough, which ensures a good rate of decay for the regularization parameter $\lambda$, leading to an asymptotic rate equivalent to the accelerated rate in \eqref{eq:rate_nesterov}.  

The {\em stability} term (in $\tilde c^\lambda - c^\lambda$) is bounded using the \emph{perturbation matrix}
\BEA
	P \triangleq R^TR - \tilde R^T \tilde R \label{eq:def_pert_matrix}
\EEA
where $R$ and $\tilde R$ are the matrices of residuals,
\BEA
 R \triangleq [r_0...r_{k}] \qquad r_t = x_{t+1}-x_t, \label{eq:def_r} \\
 \tilde R \triangleq [\tilde r_0...\tilde r_{k}] \qquad \tilde r_t = \tilde x_{t+1}-\tilde x_t. \label{eq:def_r_tilde}
\EEA
The proofs of the following propositions were obtained by \citet{scieur2016regularized}.
\begin{proposition}[Stability]
	Let $\Delta c^\lambda = \tilde c^\lambda - c^\lambda$ be the gap between the coefficients computed by Algorithm~\eqref{algo:acc_fixedpoint_reg} using the sequences $\{\tilde x_i\}$ and $\{x_i\}$ with regularization parameter $\lambda$. Let $P = R^TR-\tilde R^T\tilde R$ be defined in \eqref{eq:def_pert_matrix}, \eqref{eq:def_r} and \eqref{eq:def_r_tilde}. Then
	\BEA
		\|\Delta c^\lambda \| & \leq &   \frac{\|P\|}{\lambda} \|c^\lambda\| 
		\label{eq:bound_stability_temp}.
	\EEA
	This implies that the stability term is bounded by
	\BEA
		 \Big\|\sum_{i=0}^k\Delta c^\lambda_i(x_i-x^*) \Big\| & \leq &    \frac{\|P\|}{\lambda} \|c^\lambda\| \, O(x_0-x^*). \label{eq:bound_stability}
	\EEA
\end{proposition}

The term \textbf{Nonlinearity} is bounded by the norm of the coefficients $\tilde c_\lambda$ (controlled thanks to the regularization parameter) times the norm of the \emph{noise matrix}
\BEQ
	\mathcal{E} = [x_0-\tilde x_0, \,x_1-\tilde x_1,\, ..., \, x_k-\tilde x_k]. \label{eq:def_mathcal_e}
\EEQ

\begin{proposition}[Nonlinearity]
	Let $\tilde c^\lambda$ be computed by Algorithm \ref{algo:acc_fixedpoint_reg} using the sequence $\{\tilde x_0,\,...,\, \tilde x_{k+1}\}$ with regularization parameter $\lambda$ and $\tilde R$ be defined in \eqref{eq:def_r_tilde}. The norm of $\tilde c^\lambda$ is bounded by
	\BEQ
		\|\tilde c^\lambda\|  \leq \sqrt{\frac{\|\tilde R\|^2 + \lambda}{(k+1)\lambda}}  \leq \frac{1}{\sqrt{k+1}}\sqrt{1+\frac{\|\tilde R\|^2}{\lambda}}.\label{eq:bound_nonlinearity_temp}
	\EEQ
	This bounds the {nonlinearity} term because
	\BEA
		 \Big\|\sum_{i=0}^{k} \tilde c^{\lambda}_i (\tilde x_i-x_i) \Big\| \leq  \sqrt{1+\frac{\|\tilde R\|^2}{\lambda}} \; \frac{\|\mathcal{E}\|}{\sqrt{k+1}}, \label{eq:bound_nonlinearity}
	\EEA
	where $\mathcal{E}$ is defined in \eqref{eq:def_mathcal_e}.
\end{proposition}
These two propositions show that the regularization in Algorithm \ref{algo:acc_fixedpoint_reg} limits the impact of the noise: the higher $\lambda$ is, the smaller these terms are. It remains to control the {\em acceleration} term. We introduce the normalized regularization value $\bar \lambda$, written
\BEQ
	\bar \lambda \triangleq \frac{\lambda}{\|x_0-x^*\|^2}. \label{eq:normalized_lambda}
\EEQ
For small $\bar\lambda$, this term decreases as fast as the accelerated rate \eqref{eq:rate_nesterov}, as shown in the following proposition.

\begin{proposition}[Acceleration]
	Let $\mathcal{P}_k$ be the subspace of polynomials of degree at most $k$ and $S_\kappa(k,\alpha)$ be the solution of the Regularized Chebychev Polynomial problem,
	\BEQ
		S_\kappa(k,\alpha) \triangleq \min_{p\in \mathcal{P}_k} \;\max_{x\in[0,1-\kappa]} \; p^2(x) + \alpha \|p\|^2 \qquad s.t. \quad p(1) = 1. \label{eq:def_reg_cheby}
	\EEQ
	Let $\bar \lambda$ be the normalized value of lambda defined in \eqref{eq:normalized_lambda}. The acceleration term is bounded by
	\BEQ
		  \Big\|\sum_{i=0}^{k}  c^{\lambda}_i  x_i - x^* \Big\| \leq \frac{1}{\kappa} \sqrt{ S_{\kappa}(k,\bar \lambda) \|x_0-x^*\|^2 - \lambda \|c^\lambda\|^2 } \label{eq:bound_acceleration}.
	\EEQ
\end{proposition}

We also get the following corollary, which will be useful for the asymptotic analysis of the rate of convergence of Algorithm \ref{algo:acc_fixedpoint_reg}.

\begin{corollary} \label{cor:conv_reg_cheby}
	If $\lambda\rightarrow 0$, the bound \eqref{eq:bound_acceleration} becomes
	\[
		 \Big\|\sum_{i=0}^{k}  c^{\lambda}_i  x_i - x^* \Big\| \leq \left(\frac{1-\sqrt{\kappa}}{1+\sqrt{\kappa}}\right)^k \|x_0-x^*\|.
	\]
\end{corollary}

These last results controlling stability, nonlinearity and acceleration are proved by \citet{scieur2016regularized}. We now refine the final step of \citet{scieur2016regularized} to produce a global bound on the error that will allow us to extend these results to the stochastic setting in the next sections.

\begin{theorem}
	If Algorithm \ref{algo:acc_fixedpoint_reg} is applied to the sequence $\tilde x_i$ with regularization parameter $\lambda$, it converges with rate
	\BEQ
		\left\|\sum_{i=0}^k \tilde c_i^\lambda \tilde x_i\right\| \leq \|x_0-x^*\| S_{\kappa}(k,\bar{\lambda}) \sqrt{\frac{1}{\kappa^2} + \frac{O(\|x-x^*\|^2) \|P\|^2}{\lambda^3}} + \frac{\|\mathcal{E}\|}{\sqrt{k+1}} \sqrt{1+\frac{\|\tilde R\|^2}{\lambda}}. \label{eq:thm_deterministic}
	\EEQ
\end{theorem}
\begin{proof}
	The proof is inspired by \citet{scieur2016regularized} and is also very similar to the proof of Proposition \ref{prop:accuracy_stoch_general}. We can bound \eqref{eq:bound_acc_split} using \eqref{eq:bound_stability} ({Stability}), \eqref{eq:bound_nonlinearity} ({Nonlinearity}) and \eqref{eq:bound_acceleration} ({Acceleration}). It remains to maximize over the value of~$\|c^\lambda\|$ using the result of Proposition \ref{prop:opt_val_sqrt_fun}.
\end{proof}

This last bound is not very explicit, but an asymptotic analysis simplifies it considerably. The next new proposition shows that when $x_0$ is close to $x^*$, then extrapolation converges as fast as in \eqref{eq:rate_nesterov} in some cases.

\begin{proposition} \label{prop:assympt_rate_conv}
	Assume $\|\tilde R\| = O(\|x_0-x^*\|)$, $\|\mathcal{E}\| = O(\|x_0-x^*\|^2)$ and $\|P\| = O(\|x_0-x^*\|^3)$, which is satisfied when fixed-step gradient method is applied on a twice differentiable, smooth and strongly convex function with Lipchitz-continuous Hessian. If we chose $\lambda = O(\|x_0-x^*\|^s)$ with $s \in [2,\frac{8}{3}]$ then the bound \eqref{eq:thm_deterministic} becomes
	\[
		\lim\limits_{\|x_0-x^*\| \rightarrow 0 } \frac{\|\sum_{i=0}^k \tilde c_i^\lambda \tilde x_i\|}{\|x_0-x^*\|} \leq \frac{1}{\kappa} \left( \frac{1-\kappa}{1+\kappa} \right)^k.
	\]
\end{proposition}
\begin{proof}
The proof is based on the fact that $\lambda$ decreases slowly enough to ensure that the {\emph{Stability}} and { \emph{Nonlinearity}} terms vanish over time, but fast enough to have $\bar \lambda \rightarrow 0$. If $\lambda = \|x_0-x^*\|^s$, equation~\eqref{eq:thm_deterministic} becomes
\[
	\frac{\left\|\sum_{i=0}^k \tilde c_i^\lambda \tilde x_i\right\|}{\|x_0-x^*\|}\leq S_{\kappa}(k,\|x_0-x^*\|^{s-2}) \sqrt{\frac{1}{\kappa^2} + \frac{O(1)\|P\|^2}{\|x_0-x^*\|^{3s-2}}} + \frac{\|\mathcal{E}\|}{\|x_0-x^*\|\sqrt{k+1}} \sqrt{1+\frac{\|\tilde R\|^2}{\|x_0-x^*\|^s}}.
\]
By assumption on $\|\tilde R\|$, $\|\mathcal{E}\|$ and $\|P\|$, the previous bound becomes
\[
	\frac{\left\|\sum_{i=0}^k \tilde c_i^\lambda \tilde x_i\right\|}{\|x_0-x^*\|} \leq S_{\kappa}(k,\|x_0-x^*\|^{s-2}) \sqrt{\frac{1}{\kappa^2} + O(\|x_0-x^*\|^{8-3s})} +O\left( \sqrt{\|x_0-x^*\|^2+\|x_0-x^*\|^{4-s}}\right).
\]
Because $\lambda \in ]2,\frac{8}{3}[$, all exponents of $\|x_0-x^*\|$ are positive. By consequence, when 
\[
	\lim\limits_{\|x_0-x^*\|\rightarrow 0} \frac{\left\|\sum_{i=0}^k \tilde c_i^\lambda \tilde x_i\right\|}{\|x_0-x^*\|} \leq \frac{1}{\kappa} S_{\kappa}(k,0).
\]
Finally, the desired result is obtained by using Corollary \ref{cor:conv_reg_cheby}.
\end{proof}

The efficiency of Algorithm \ref{algo:acc_fixedpoint_reg} is thus ensured by two conditions. First, we need to be able to bound $\|\tilde R\|$, $\|P\|$ and $\|\mathcal{E}\|$ by decreasing quantities. Second, we have to find a proper rate of decay for $\lambda$ and $\bar \lambda$ such that the {\em stability} and {\em nonlinearity} terms go to zero when perturbations also go to zero. If these two conditions are met, then the accelerated rate in Proposition \ref{prop:assympt_rate_conv} holds.

\section{Nonlinear and Noisy Updates}
In \eqref{eq:fixed_point_iteration} we defined $g(x)$ to be non linear, which generates a sequence $\tilde x_i$. We now consider noisy iterates
\BEQ
	\tilde x_{t+1} = g(\tilde x_t) +  \eta_{t+1} \label{eq:temp_stoch_update}
\EEQ
where $\eta_t$ is a stochastic noise. To simplify notations, we write \eqref{eq:temp_stoch_update} as
\BEQ
	\tilde x_{t+1} = x^* + G(\tilde x_t-x^*) + \varepsilon_{t+1}, \label{eq:stoch_update}
\EEQ
where $\varepsilon_t$ is a stochastic noise (potentially correlated with the iterates $x_i$) with bounded mean $\nu_t$, $\|\nu_t\|\leq \nu$ and bounded covariance $\Sigma_t\preceq(\sigma^2/d)I$. We also assume $\textbf{0}\preceq G \preceq (1-\kappa)\textbf{I}$ and $G$~symmetric. For example,~\eqref{eq:stoch_update} can be linked to  \eqref{eq:temp_stoch_update} if we set $\varepsilon_t = \eta_t + O(\|\tilde x_t-x^*\|^2)$. This corresponds to the combination of the noise $\eta_{t+1}$ with the Taylor remainder of $g(x)$ around $x^*$.

The recursion \eqref{eq:stoch_update} is also valid when we apply the stochastic gradient method to the quadratic problem 
\[
	 \min_x \frac{1}{2}\|Ax-b\|^2.
\]
This correspond to \eqref{eq:stoch_update} with  $G = I-hA^TA$ and mean $\nu = 0$. For the theoretical results, we will compare $\tilde x_t$ with their noiseless counterpart to control convergence,
\BEQ
	x_{t+1} = x^* + G(x_t-x^*), \quad x_0 = \tilde x_0. \label{eq:noiseless_update}
\EEQ
\section{Convergence Analysis when Accelerating Stochastic Algorithms}

We will control convergence in expectation. Bound \eqref{eq:bound_acc_split} now becomes
\BEQ
	\E \left[ \Big\| \sum_{i=0}^{k} \tilde c^{\lambda}_i \tilde x_i - x^*\Big\|\right] \leq 
	\Big\|\sum_{i=0}^{k}  c^{\lambda}_i  x_i - x^* \Big\| 
	+ O(\|x_0-x^*\|)\E \Big[ \|\Delta c^\lambda\|\Big]
	+  \E \Big[\|\tilde c^\lambda\| \|\mathcal{E}\|\Big]. \label{eq:bound_stochastic}
\EEQ
We now need to enforce bounds \eqref{eq:bound_stability}, \eqref{eq:bound_nonlinearity}  and \eqref{eq:bound_acceleration} in expectation. For simplicity, we will omit all constants in what follows.
\begin{proposition} \label{prop:stoch_bound}
	Consider the sequences $x_i$ and $\tilde x_i$ generated by \eqref{eq:temp_stoch_update} and \eqref{eq:noiseless_update}. Then,
	\BEA
		\E [\|\tilde R\| ]& \leq & O(\|x_0-x^*\|) + O(\nu + \sigma)
		\label{eq:bound_stoch_tilde_R} \\
		\E [\|\mathcal{E}\|] & \leq & O(\nu + \sigma) \label{eq:bound_stoch_mathcal_E} \\
		\E [\|P\|] & \leq & O((\sigma + \nu)\|x_0-x^*\|) + O((\nu + \sigma)^2).
	\EEA
\end{proposition}
\begin{proof}
	First, we have to form the matrices $\tilde R$, $\mathcal{E}$ and $P$. We begin with $\mathcal{E}$, defined in \eqref{eq:def_mathcal_e}. Indeed,
	\BEAS
	\mathcal{E}_i =  x_i-\tilde x_i  &  \Rightarrow  & \mathcal{E}_1 = \varepsilon_1, \\
	& & \mathcal{E}_2 = \varepsilon_2 + G \varepsilon_1, \\
	& & \mathcal{E}_k = \sum_{i=1}^k G^{k-i}\varepsilon_i.
	\EEAS	
	It means that each $\|\mathcal{E}_i\| = O(\|\varepsilon_i\|)$. By using \eqref{eq:jensen_ineq_matrix},
	\BEAS
		\E\|\mathcal{E}\| & \leq & \sum_i \E\|\mathcal{E}_i\|\\
		& \leq & \sum_i \E\|\mathcal{E}_i-\nu_i\| + \|\nu_i\| \\
		& \leq & O(\nu + \sigma)
	\EEAS
	For $\tilde R$, we notice that
	\BEAS
	\tilde R_t & = & \tilde x_{t+1} - \tilde x_t, \\
	& = & R_t + \mathcal{E}_{t+1}-\mathcal{E}_{t}
	\EEAS
	We get \eqref{eq:bound_stoch_tilde_R} by splitting the norm,
	\[
	\E [ \|\tilde R \| ] \leq \|R\| + O(\|\mathcal{E}\|)  \leq O(\|x_0-x^*\|) + O\left(\nu + \sigma\right).
	\]
	Finally, by definition of $P$,
	\[
	\|P\| \leq 2\|\mathcal{E}\|\|R\| + \|\mathcal{E}\|^2.
	\]
	Taking the expectation leads to the desired result,
	\BEAS
	\E [\|P\|] & \leq & 2 \E [\|\mathcal{E}\|\|R\|] + \E[ \|\mathcal{E}\|^2], \\
	& \leq & 2 \|R\| \, \E [\|\mathcal{E}\|] + \E[ \|\mathcal{E}\|_F^2],\\
	& \leq & O\left(\|x_0-x^*\|(\sigma + \nu)\right) + O\left((\sigma+\nu)^2\right).
	\EEAS
\end{proof}

We define the following \emph{stochastic condition number}
\[
	\tau \triangleq \frac{\nu + \sigma}{\|x_0-x^*\|}.
\] 
The Proposition \ref{prop:accuracy_stoch_general} gives the result when injecting these bounds in \eqref{eq:bound_stochastic}.

\begin{proposition} \label{prop:accuracy_stoch_general}
The accuracy of extrapolation Algorithm \ref{algo:acc_fixedpoint_reg} applied to the sequence $\{\tilde x_0,...,\tilde x_k\}$ generated by \eqref{eq:temp_stoch_update} is bounded by
	\BEA
		\frac{\E \Big[ \| \sum_{i=0}^{k} \tilde c^{\lambda}_i \tilde x_i - x^*\|\Big]}{ \|x_0-x^*\|} \leq \left( S_{\kappa}(k,\bar{\lambda}) \sqrt{\frac{1}{\kappa^2} + \frac{O(\tau^2(1+\tau)^2)}{\bar \lambda^3}} + O\Bigg(\sqrt{ \tau^2+\frac{ \tau^2(1+ \tau^2)}{\bar \lambda} }\Bigg)\right).
		 \label{eq:bound_stochastic_full}
	\EEA
\end{proposition}
\begin{proof}
	We start with \eqref{eq:bound_stochastic}, then we use \eqref{eq:bound_stability_temp}
	\BEAS
	& &	\Big\|\sum_{i=0}^{k}  c^{\lambda}_i  x_i - x^* \Big\| 
	+ O(\|x_0-x^*\|)\E \Big[ \|\Delta c^\lambda\|\Big]
	+ \E \Big[\|\tilde c^\lambda\| \|\mathcal{E}\|\Big], \\
	& \leq & \Big\|\sum_{i=0}^{k}  c^{\lambda}_i  x_i - x^* \Big\| 
	+ O(\|x_0-x^*\|)\frac{\|c_\lambda\|}{\lambda}\E \Big[\|P\|   \Big]
	+ \sqrt{\E \Big[ \|\tilde c_\lambda \|^2\Big]\E\Big[ \|\mathcal{E}\|^2\Big]}.
	\EEAS
	The first term can be bounded by \eqref{eq:bound_acceleration},
	\[
	\Big\|\sum_{i=0}^{k}  c^{\lambda}_i  x_i - x^* \Big\| \leq \frac{1}{\kappa} \sqrt{ S_{\kappa}(k,\bar \lambda) \|x_0-x^*\|^2 - \lambda \|c^\lambda\|^2 }.
	\]
	We combine this bound with the second term by maximizing over $\|c^\lambda\|$. The optimal value is given in \eqref{eq:opt_val_sqrt_fun},
	\[
	\Big\|\sum_{i=0}^{k}  c^{\lambda}_i  x_i - x^* \Big\| 
	+ O(\|x_0-x^*\|)\frac{\|c_\lambda\|}{\lambda}\E \Big[\|P\|   \Big] \leq \|x_0-x^*\| S_{\kappa}(k,\bar{\lambda}) \sqrt{\frac{1}{\kappa^2} + \frac{O(\|x-x^*\|^2) \E [\|P\|]^2}{\lambda^3}},
	\]
	where $\bar \lambda = \lambda / \|x_0-x^*\|^2$.	Since, by Proposition \ref{prop:stoch_bound},
	\[
	\E [\|P\|]^2 \leq  O\left((\nu + \sigma)^2 \left(\|x_0-x^*\|+\nu + \sigma\right)^2\right),
	\]
	we have
	\BEA
	& & \|\sum_{i=0}^{k}  c^{\lambda}_i  x_i - x^* \| 
	+ O(\|x_0-x^*\|)\frac{\|c_\lambda\|}{\lambda}\E \Big[\|P\|   \Big] \nonumber \\
	& \leq & \|x_0-x^*\| S_{\kappa}(k,\bar{\lambda}) \sqrt{\frac{1}{\kappa^2} + \frac{O(\|x-x^*\|^2(\nu+\sigma)^2)(\|x_0-x^*\|+\nu + \sigma)^2}{\lambda^3}}. \label{eq:temp_stoch_1}
	\EEA
	The last term can be bounded using \eqref{eq:bound_nonlinearity_temp},
	\BEAS
	\sqrt{\E \Big[ \|c_\lambda \|^2\Big]\E\Big[ \|\mathcal{E}\|^2\Big]}
	& \leq & O\Big(\Big(\sum_{i=0}^k \|\mathcal{E}\|_i^2\E \Big[ \|\tilde c_\lambda \|^2\Big] \Big)^{1/2}\Big)\\
	& \leq & O\Big((\nu+\sigma)\sqrt{\E \Big[ \|\tilde c_\lambda \|^2\Big] }\Big)\\
	& \leq & O\Big((\nu+\sigma)\sqrt{\E \Big[ 1+\frac{\|\tilde R\|^2}{\lambda}\Big] }\Big)\\
	& \leq & O\Bigg((\nu+\sigma)\sqrt{ 1+\frac{\E \big[\|\tilde R\|_F^2\big]}{\lambda} }\Bigg)		
	\EEAS
	However,
	\BEAS
	\textstyle \E \Big[\|\tilde R\|_F^2\Big] & = & \textstyle \sum_{i=0}^k\E \Big[\|\tilde r_i\|^2\Big] \\
	& = & \textstyle \sum_{i=0}^k \|r_i\|^2 + \E \Big[ r_i^T\mathcal{E}_i + \|\mathcal{E}_i\|^2\Big] \\
	& \leq & O\left( \|x_0-x^*\|^2 + (\nu+\sigma) \|x_0-x^*\| + (\nu+\sigma)^2 \right)\\
	& \leq & O\left( \|x_0-x^*\| + (\nu+\sigma) )^2 \right)
	\EEAS
	Finally,
	\BEQ
	\sqrt{\E \Big[ \|c_\lambda \|^2\Big]\E\Big[ \|\mathcal{E}\|^2\Big]} \leq O\Bigg((\nu+\sigma)\sqrt{ 1+\frac{ (\|x_0-x^*\| + (\nu+\sigma))^2}{\lambda} }\Bigg) \label{eq:temp_stoch_2}
	\EEQ
	We get \eqref{eq:bound_stochastic_full} by summing \eqref{eq:temp_stoch_1} and \eqref{eq:temp_stoch_2}, then by replace all $\frac{\nu + \sigma}{\|x_0-x^*\|}$ by $\tau$ and $\frac{\lambda}{\|x_0-x^*\|^2}$ by $\bar \lambda$.
\end{proof}

Consider a situation where $\tau$ is small, e.g. when using stochastic gradient descent with fixed step-size, with $x_0$ far from $x^*$. The following proposition details the dependence between $\bar \lambda$ and $\tau$ ensuring the upper convergence bound remains stable when $\tau$ goes to zero.

\begin{proposition}\label{prop:asympt_perf_acc}
	 When $\tau \rightarrow 0$, if $\bar \lambda = \Theta(\tau^{s}) $ with $s\in ]0,\frac{2}{3}[$, we have the accelerated rate
	\BEA
		 \E \bigg[ \Big\| \sum_{i=0}^{k} \tilde c^{\lambda}_i \tilde x_i - x^* \Big\|\bigg] \leq  \frac{1}{\kappa}\left(\frac{1-\sqrt{\kappa}}{1+\sqrt{\kappa}}\right)^k\|x_0-x^*\|. \label{eq:asympt_result}
	\EEA
	Moreover, if $\lambda \rightarrow \infty$, we recover the averaged gradient,
	\[
		 \E \bigg[ \Big\| \sum_{i=0}^{k} \tilde c^{\lambda}_i \tilde x_i - x^* \Big\|\bigg] = \E \bigg[ \Big\| \frac{1}{k+1}\sum_{i=0}^{k} \tilde x_i - x^*\Big\|\bigg].
	\]
\end{proposition}
\begin{proof}
	Let $\bar{\lambda} = \Theta(\tau^s)$, using \eqref{eq:bound_stochastic_full} we have
\BEAS
		 \E \bigg[ \Big\| \sum_{i=0}^{k} \tilde c^{\lambda}_i \tilde x_i - x^*\Big\|\bigg] &\leq&
		 \|x_0-x^*\|  S_{\kappa}(k,\tau^s) \sqrt{\frac{1}{\kappa^2} O(\tau^{2-3s}(1+\tau)^2)} \\
		& &  + \|x_0-x^*\|O(\sqrt{ \tau^2+\tau^{2-3s}(1+ \tau^2) }).
\EEAS
	Because $s\in ]0,\frac{2}{3}[$, means $2-3s>0$, thus $\lim_{\tau\rightarrow 0} \tau^{2-3s} = 0$. The limits when $\tau \rightarrow 0$ is thus exactly \eqref{eq:asympt_result}. If $\lambda \rightarrow \infty$, we have also
	\[
		 \lim\limits_{\lambda \rightarrow\infty}\tilde c_\lambda \; = \;  \lim\limits_{\lambda \rightarrow\infty} \argmin_{c:\ones^Tc = 1} \|\tilde Rc\|  + \lambda \|c\|^2 \; = \;  \argmin_{c:\ones^Tc = 1} \|c\|^2 \; = \; \frac{\ones}{k+1}
	\] which yields the desired result.
\end{proof}

Proposition \ref{prop:asympt_perf_acc} shows that Algorithm \ref{algo:acc_fixedpoint_reg} is thus asymptotically optimal provided $\lambda$ is well chosen because it recovers the accelerated rate for smooth and strongly convex functions when the perturbations goes to zero. Moreover, we recover Proposition \ref{prop:assympt_rate_conv} when $\epsilon_t$ is the Taylor remainder, i.e. with $\nu = O(\|x_0-x^*\|^2)$ and $\sigma = 0$, which matches the deterministic results.

Algorithm \ref{algo:acc_fixedpoint_reg} is particularly efficient when combined with a restart scheme \citep{scieur2016regularized}. From a theoretical point of view, the acceleration peak arises for small values of $k$. Empirically, the improvement is usually more important at the beginning, i.e. when $k$ is small. Finally, the algorithmic complexity is $O(k^2d)$, which is linear in the problem dimension when $k$ remains bounded. 

The benefits of extrapolation are limited in a regime where the noise dominates. However, when $\tau$ is relatively small then we can expect a significant speedup. This condition is satisfied in many cases, for example at the initial phase of the stochastic gradient descent or when optimizing a sum of functions with variance reduction techniques, such as SAGA or SVRG.

\section{Numerical Experiments}

\subsection{Stochastic gradient descent}
We want to solve the least-square problem
\[
	\min_{x\in R^d} F(x) = \frac{1}{2} \|Ax-b\|^2
\]
where $A^TA$ satisfies $\mu I \preceq (A^TA) \preceq LI$. To solve this problem, we have access to the stochastic first-order oracle
\[
	\nabla_\varepsilon F(x) = \nabla F(x) + \varepsilon,
\]
where $\varepsilon$ is a zero-mean noise of covariance matrix $\Sigma \preceq \frac{\sigma^2}{d}\textbf{I}$. We will compare several methods.
\begin{itemize}\itemsep 0ex
\item	 \textbf{SGD.} Fixed step-size, $x_{t+1} = x_t-\frac{1}{L}\nabla_\varepsilon F(x_t)$.
\item	 \textbf{Averaged SGD.} Iterate $x_k$ is the mean of the $k$ first iterations of SGD.
\item	 \textbf{AccSGD.} The optimal two-step algorithm in \citet{flammarion2015averaging}, with optimal parameters (this implies $\|x_0-x^*\|$ and $\sigma$ are known exactly).
\item	 \textbf{RNA+SGD.} The regularized nonlinear acceleration Algorithm \ref{algo:acc_fixedpoint_reg} applied to a sequence of $k$ iterates of SGD, with $k = 10$ and $\lambda = \|\tilde R^T \tilde R\|/10^{-6}$.
\end{itemize}
	 
By Proposition \ref{prop:accuracy_stoch_general}, we know that RNA+SGD will \emph{not} converge to arbitrary precision because the noise is additive with a non-vanishing variance. However, Proposition \ref{prop:asympt_perf_acc} predicts an improvement of the convergence at the beginning of the process. We illustrate this behavior in Figure \ref{fig:exp_sgd}. We clearly see that at the beginning, the performances of RNA+SGD is comparable to that of the optimal accelerated algorithm. However, because of the restart strategy, in the regime where the level of noise becomes more important the acceleration becomes less effective and finally the convergence stalls, as for SGD. Of course, for practical purposes, the first regime is the most important because it effectively minimizes the generalization error \citep{defossez2015averaged,jain2016parallelizing}.

\begin{figure}[t]
	\centering
	
	\psfrag{sgd}{\footnotesize SGD}
	\psfrag{avesgd}{\footnotesize Ave. SGD}
	\psfrag{accsgd}{\footnotesize Acc. SGD}
	\psfrag{rmpesgd}{\footnotesize{ RNA + SGD}}
	\includegraphics[width=0.7\textwidth]{figs/legend_sgd_ls.eps}
	~\\
	
	\psfrag{xlabel}{\scriptsize Iteration}
	\psfrag{valf}[b]{\scriptsize $f(x)-f(x^*)$}
	\includegraphics[width=0.33\textwidth]{figs/dist_1e4_sigma_10_d_300_kappa_2.eps}
	\psfrag{valf}[b]{\scriptsize }
	\includegraphics[width=0.33\textwidth]{figs/dist_1e4_sigma_1000_d_300_kappa_2.eps}
	\psfrag{valf}[b]{\scriptsize }
	\includegraphics[width=0.33\textwidth]{figs/dist_1e4_sigma_1000_d_300_kappa_6.eps}
	\caption*{\textbf{Left:} $\sigma = 10$, $\kappa = 10^{-2}$. \textbf{Center:} $\sigma = 1000$, $\kappa = 10^{-2}$. \textbf{Right:} $\sigma = 1000$, $\kappa = 10^{-6}$. }
	~\\
	
	\psfrag{valf}[b]{\scriptsize $f(x)-f(x^*)$}
	\includegraphics[width=0.33\textwidth]{figs/dist_1e4_sigma_10_d_300_decay_1n.eps}
	\psfrag{valf}[b]{\scriptsize }
	\includegraphics[width=0.33\textwidth]{figs/dist_1e4_sigma_100_d_300_decay_1n.eps}
	\psfrag{valf}[b]{\scriptsize }
	\includegraphics[width=0.33\textwidth]{figs/dist_1e4_sigma_1000_d_300_decay_1n.eps}
	
	\caption*{\textbf{Left:} $\sigma = 10$, $\kappa = 1/d$. \textbf{Center:} $\sigma = 100$, $\kappa = 1/d$. \textbf{Right:} $\sigma = 1000$, $\kappa = 1/d$.}
	~\\

\caption{Comparison of performances between SGD, averaged SGD, Accelerated SGD \citep{flammarion2015averaging} and RNA+SGD. We tested the performances on a matrix $A^TA$ of size $d=500$, with (\textbf{top}) random eigenvalues between $\kappa$ and $1$ and (\textbf{bottom}) decaying eigenvalues from $1$ to $1/d$. We start at $\|x_0-x^*\| = 10^{4}$, where $x_0$ and $x^*$ are generated randomly. \newline }
	\label{fig:exp_sgd}
\end{figure}

\subsection{Finite sums of functions}
We focus on the composite problem $\min_{x\in R^d}  F(x) = \sum_{i=1}^{N} \frac{1}{N} f_i(x) + \frac{\mu}{2} \|x\|^2$, where $f_i$ are convex and $L$-smooth functions and $\mu$ is the regularization parameter. We will use classical methods for minimizing $F(x)$ such as SGD (with fixed step size), SAGA \citep{defazio2014saga}, SVRG \citep{johnson2013accelerating}, and also the accelerated algorithm Katyusha \citep{allen2016katyusha}. We will compare their performances with and without the (potential) acceleration provided by Algorithm \ref{algo:acc_fixedpoint_reg} with restart each $k$ iteration. The parameter $\lambda$ is found by a grid search of size $k$, the size of the input sequence, but it adds only \emph{one} data pass at each extrapolation. Actually, the grid search can be faster if we approximate $F(x)$ with fewer samples, but we choose to present Algorithm~\ref{algo:acc_fixedpoint_reg} in its simplest version. We set $k = 10$ for all the experiments.

In order to balance the complexity of the extrapolation algorithm and the optimization method we wait several data queries before adding the current point (the ``snapshot'') of the method to the sequence. Indeed, the extrapolation algorithm has a complexity of $O(k^2d) + O(N)$ (computing the coefficients $\tilde c^\lambda$ and the grid search over $\lambda$). If we wait at least $O(N)$ updates, then the extrapolation method is of the same order of complexity as the optimization algorithm.

\begin{itemize}\itemsep 0ex
\item	 \textbf{SGD.} We add the current point after $N$ data queries (i.e. one epoch) and $k$ snapshots of SGD cost $kN$ data queries.
\item	 \textbf{SAGA.} We compute the gradient table exactly, then we add a new point after $N$ queries, and $k$ snapshots of SAGA cost $(k+1)N$ queries. Since we optimize a sum of quadratic or logistic losses, we used the version of SAGA which stores $O(N)$ scalars.
\item	 \textbf{SVRG.} We compute the gradient exactly, then perform $N$ queries (the inner-loop of SVRG), and $k$ snapshots of SVRG cost $2kN$ queries.
\item	 \textbf{Katyusha.} We compute the gradient exactly, then perform $4N$ gradient calls (the inner-loop of Katyusha), and $k$ snapshots of Katyusha cost $3kN$ queries.
\end{itemize}

We compare these various methods for solving least-square regression and logistic regression on several datasets (Table \ref{table:datasets}), with several condition numbers $\kappa$: well ($\kappa = 100/N$), moderately ($\kappa = 1/N$) and badly ($\kappa = 1/100N$) conditioned. In this section, we present the numerical results on \texttt{Sid} (Sido0 dataset, where $N = 12678$ and $d = 4932$) with bad conditioning, see Figure \ref{fig:exp_sido0}. The other experiments are highlighted in the supplementary material.

In Figure \ref{fig:exp_sido0}, we clearly see that both SGD and AccSGD do not converge. This is mainly due to the fact that we do not average the points. In any case, except for quadratic problems, the averaged version of SGD does not converge to the minimum of $F$ with arbitrary precision.

We also notice that Algorithm \ref{algo:acc_fixedpoint_reg} is unable to accelerate Katyusha. This issue was already raised by \citet{scieur2016regularized}: when the algorithm has a momentum term (like the Nesterov's method), the underlying dynamical system is harder to extrapolate.

Because the iterates of SAGA and SVRG have low variance, their accelerated version converges faster to the optimum, and their performances are then comparable to Katyusha. In our experiments, Katyusha was faster than AccSAGA only once, when solving a least square problem on Sido0 with a bad condition number. Recall however that the acceleration Algorithm~\ref{algo:acc_fixedpoint_reg} does not require the specification of the strong convexity parameter, unlike Katyusha.

\begin{figure}
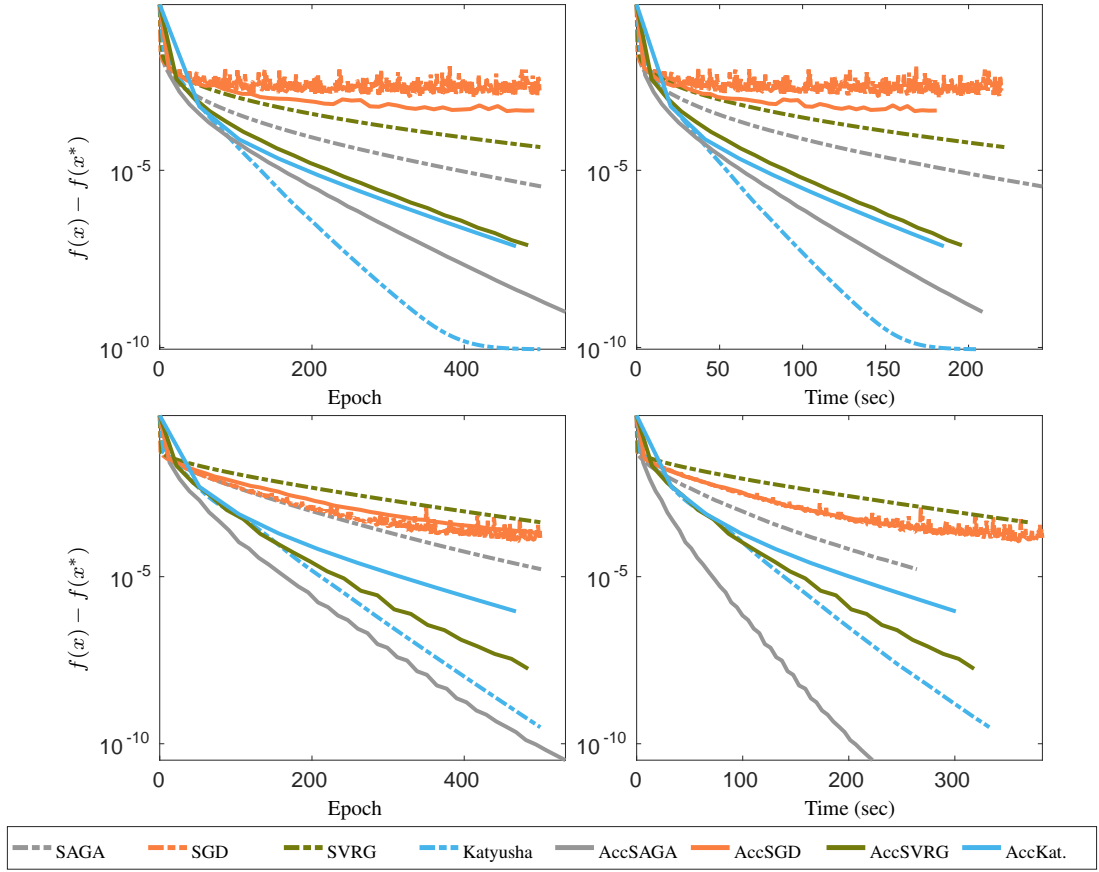

	\centering
	
	\psfrag{valf}[b]{\scriptsize $f(x)-f(x^*)$}
	\psfrag{xlabel}{\scriptsize Epoch}
		\includegraphics[width=0.40\textwidth]{figs/classif/leastsquare_sido0_bad_cond_12678_4932_nite.eps}
		\psfrag{xlabel}{\scriptsize Time (sec)}
		\psfrag{valf}[b]{\scriptsize }
		\includegraphics[width=0.40\textwidth]{figs/classif/leastsquare_sido0_bad_cond_12678_4932_time.eps}
	~\\
	\psfrag{valf}[b]{\scriptsize $f(x)-f(x^*)$}
	\psfrag{xlabel}{\scriptsize Epoch}
	\includegraphics[width=0.40\textwidth]{figs/classif/logistic_sido0_bad_cond_12678_4932_nite.eps}
	\psfrag{xlabel}{\scriptsize Time (sec)}
	\psfrag{valf}[b]{\scriptsize }
	\includegraphics[width=0.40\textwidth]{figs/classif/logistic_sido0_bad_cond_12678_4932_time.eps}
	\psfrag{data1}{\tiny SAGA}
	\psfrag{data2}{\tiny SGD}
	\psfrag{data3}{\tiny SVRG}
	\psfrag{data4}{\tiny Katyusha}
	\psfrag{data5}{\tiny AccSAGA}
	\psfrag{data6}{\tiny AccSGD}
	\psfrag{data7}{\tiny AccSVRG}
	\psfrag{data8}{\tiny AccKat.}
	\includegraphics[width=0.9\textwidth]{figs/classif/legend.eps}
	\caption{Optimization of quadratic loss (\textbf{Top}) and logistic loss (\textbf{Bottom}) with several algorithms, using the \texttt{Sid} dataset with bad conditioning. The experiments are done in Matlab. \textbf{Left:} Error vs epoch number. \textbf{Right:} Error vs time. }
	\label{fig:exp_sido0}
\end{figure}

\subsection{Neural Networks} In this section we will use the RNA algorithm with $k=15$ and $\lambda=10^{-5}$ in order to reduce  the training time of neural networks. We focus on the training loss value and the test error (percentage of bad predictions). We tested the acceleration on several (deep) neural networks for image classification. In all the experiments we use the classical SGD + momentum (set at 0.9) optimization algorithm, with a weight decay of \texttt{1e-5}. The learning rate changes in function of the neural network, and will be specified in each section. Unless it is specified, we do not use the restart strategy.

This time, we do not use the adaptive value of $\lambda$, because performing a complete data pass with neural networks may take a lot of time. However, as mentioned above, one can use a portion of the samples instead of the complete dataset to estimate the effectiveness of a given regularization parameter. But again, we choose to use the RNA algorithm in its simplest version.

\subsubsection{LeNet}
We first use the standard network of \citet{lecun1998gradient} for classifying images from MNIST (Figure~\ref{fig:lenet_mnist}) and CIFAR10 (Figure \ref{fig:lenet_cifar10}) datasets. We used a learning rate of $10^{-2}$. The network performs well in practice on MNIST, so we do not see acceleration. However, we clearly have an important improvement on CIFAR10 and the extrapolated network performs a better generalization error than the original one.

\begin{figure}[h!]
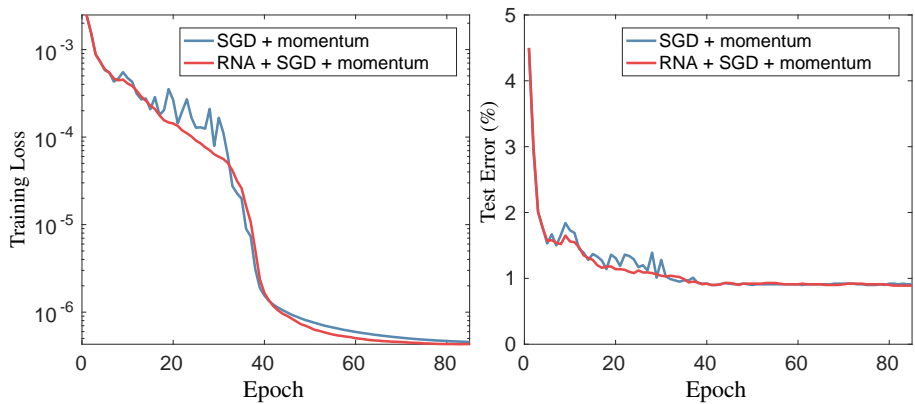

	\centering
	\psfrag{TrainingLoss}{\scriptsize Training Loss}
	\psfrag{epoch}{\footnotesize Epoch}
	\includegraphics[width=0.45\textwidth]{figs/nn/training_error_LeNet_MNIST_SGDmomentum_rna_sliding.eps}
	\psfrag{TestAcc}{\scriptsize Test Error (\%)}
	\psfrag{epoch}{\footnotesize Epoch}
	\includegraphics[width=0.45\textwidth]{figs/nn/test_error_LeNet_MNIST_SGDmomentum_rna_sliding.eps}
	\caption{Training LeNet on MNIST dataset}
	\label{fig:lenet_mnist}
\end{figure}
\begin{figure}[h!]
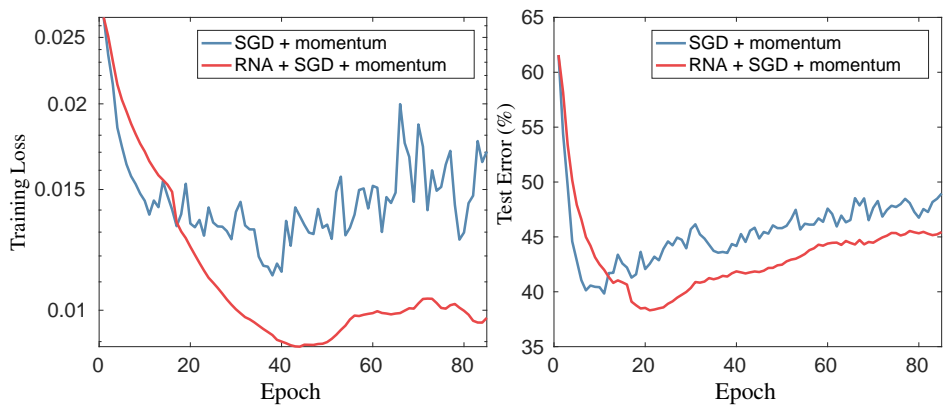

	\centering
	\psfrag{TrainingLoss}{\scriptsize Training Loss}
	\psfrag{epoch}{\footnotesize Epoch}
	\includegraphics[width=0.45\textwidth]{figs/nn/training_error_LeNet_CIFAR10_SGDmomentum_rna_sliding.eps}
	\psfrag{TestAcc}{\scriptsize Test Error (\%)}
	\psfrag{epoch}{\footnotesize Epoch}
	\includegraphics[width=0.45\textwidth]{figs/nn/test_error_LeNet_CIFAR10_SGDmomentum_rna_sliding.eps}
	\caption{Training LeNet on CIFAR10 dataset}
	\label{fig:lenet_cifar10}
\end{figure}

\subsubsection{MobileNet}
We also applied the acceleration algorithm on MobileNet \citep{howard2017mobilenets} with a learning rate of $10^{-2}$. For this network, the training loss of the accelerated version decays much faster than the original network. \newpage

\begin{figure}[h!]
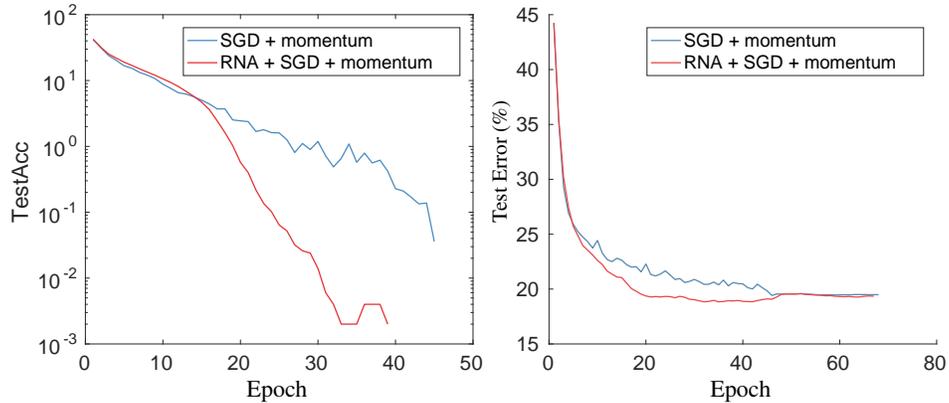

	\centering
	\psfrag{TestAcc}{\scriptsize Training Loss}
	\psfrag{epoch}{\footnotesize Epoch}
	\includegraphics[width=0.45\textwidth]{figs/nn/training_error_MobileNet_CIFAR10_SGDmomentum_rna_sliding.eps}
	\psfrag{TestAcc}{\scriptsize Test Error (\%)}
	\psfrag{epoch}{\footnotesize Epoch}
	\includegraphics[width=0.45\textwidth]{figs/nn/test_error_MobileNet_CIFAR10_SGDmomentum_rna_sliding.eps}
	\caption{Training MobileNet on CIFAR10 dataset}
	\label{fig:mobilenet_cifar10}
\end{figure}

\subsubsection{Wide Residual Networks}
We will compare the results of the acceleration of shallow and deep residual networks \citep{zagoruyko2016wide}. This time, the learning rate starts at $10^{-1}$ and decays over time by a ratio of $\frac{1}{5}$ after $\{60,120,160\}$ data passes (the update of the learning rate is shown with the dotted lines in the figures). We used Resnet-10-1 on CIFAR10 (Figure \ref{fig:resnet_cifar10}) and CIFAR100 (Figure \ref{fig:resnet_cifar100}), as well as Resnet-28-10 on CIFAR10 (Figure \ref{fig:resnet_deep_cifar10}) and CIFAR100 (Figure \ref{fig:resnet_deep_cifar100}). The last network achieves a test accuracy  of 81.72\% of correct answers, only after 82 epochs (without dropout), meanwhile the unaccelerated network achieves 80,75\% after 200 epochs (up to 81.15 if we also use dropout).

On these figures, we clearly see an important improvement both in term of training loss and test accuracy when using the extrapolated network. Moreover, the learning rate can be updated earlier if we use the extrapolation step as the new starting point. This strategy is illustrated in Figure \ref{fig:resnet_restart_cifar10}, where the learning rate decays at a rate of $\sqrt{1/5}$ every 25 epochs.

\begin{figure}[h!]
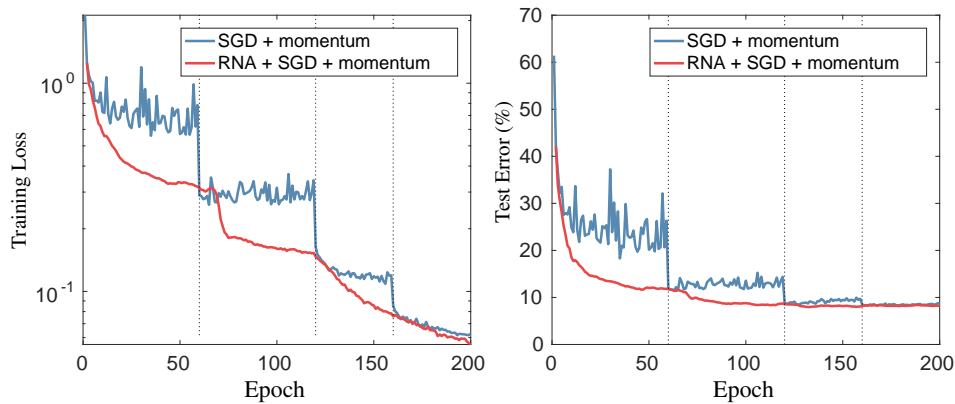

	\centering
	\psfrag{TrainingLoss}{\scriptsize Training Loss}
	\psfrag{epoch}{\footnotesize Epoch}
	\includegraphics[width=0.45\textwidth]{figs/nn/training_error_resnet_CIFAR10_SGDmomentum_rna_sliding.eps}
	\psfrag{TestAcc}{\scriptsize Test Error (\%)}
	\psfrag{epoch}{\footnotesize Epoch}
	\includegraphics[width=0.45\textwidth]{figs/nn/test_error_resnet_CIFAR10_SGDmomentum_rna_sliding.eps}
	\caption{Training Resnet-10-1 on CIFAR10 dataset}
	\label{fig:resnet_cifar10}
\end{figure}
\begin{figure}[h!]
	\centering
	\psfrag{TrainingLoss}{\scriptsize Training Loss}
	\psfrag{epoch}{\footnotesize Epoch}
	\includegraphics[width=0.45\textwidth]{figs/nn/training_error_resnet_CIFAR100_SGDmomentum_rna_sliding.eps}
	\psfrag{TestAcc}{\scriptsize Test Error (\%)}
	\psfrag{epoch}{\footnotesize Epoch}
	\includegraphics[width=0.45\textwidth]{figs/nn/test_error_resnet_CIFAR100_SGDmomentum_rna_sliding.eps}
	\caption{Training Resnet-10-1 on CIFAR100 dataset}
	\label{fig:resnet_cifar100}
\end{figure}
\begin{figure}[h!]
	\centering
	\psfrag{TrainingLoss}{\scriptsize Training Loss}
	\psfrag{epoch}{\footnotesize Epoch}
	\includegraphics[width=0.45\textwidth]{figs/nn/training_error_resnet_deep_CIFAR10_SGDmomentum_rna_sliding.eps}
	\psfrag{TestAcc}{\scriptsize Test Error (\%)}
	\psfrag{epoch}{\footnotesize Epoch}
	\includegraphics[width=0.45\textwidth]{figs/nn/test_error_resnet_deep_CIFAR10_SGDmomentum_rna_sliding.eps}
	\caption{Training Resnet-28-10 on CIFAR10 dataset}
	\label{fig:resnet_deep_cifar10}
\end{figure}
\begin{figure}[h!]
	\centering
	\psfrag{TrainingLoss}{\scriptsize Training Loss}
	\psfrag{epoch}{\footnotesize Epoch}
	\includegraphics[width=0.45\textwidth]{figs/nn/training_error_resnet_deep_CIFAR100_SGDmomentum_rna_sliding.eps}
	\psfrag{TestAcc}{\scriptsize Test Error (\%)}
	\psfrag{epoch}{\footnotesize Epoch}
	\includegraphics[width=0.45\textwidth]{figs/nn/test_error_resnet_deep_CIFAR100_SGDmomentum_rna_sliding.eps}
	\caption{Training Resnet-28-10 on CIFAR100 dataset}
	\label{fig:resnet_deep_cifar100}
\end{figure}
\begin{figure}[h!]
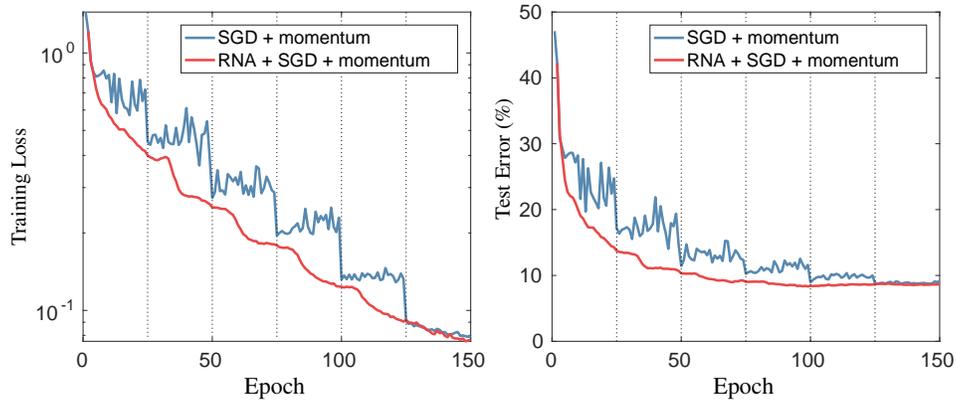

	\centering
	\psfrag{TrainingLoss}{\scriptsize Training Loss}
	\psfrag{epoch}{\footnotesize Epoch}
	\includegraphics[width=0.45\textwidth]{figs/nn/training_error_resnet_CIFAR10_SGDmomentum_rna_restart.eps}
	\psfrag{TestAcc}{\scriptsize Test Error (\%)}
	\psfrag{epoch}{\footnotesize Epoch}
	\includegraphics[width=0.45\textwidth]{figs/nn/test_error_resnet_CIFAR10_SGDmomentum_rna_restart.eps}
	\caption{Training Resnet-10-1 with a restart strategy. When updating the learning rate, we use the extrapolated version of the network as the new starting point.}
	\label{fig:resnet_restart_cifar10}
\end{figure}

\section{Acknowledgments}
The research leading to these results has received funding from the European Union’s Seventh Framework Programme (FP7-PEOPLE-2013-ITN) under grant agreement no 607290 SpaRTaN, as well as support from ERC SIPA and the chaire \'Economie des nouvelles donn\'ees with the data science joint research initiative with the fonds AXA pour la recherche.

\section{Appendix}
\subsection{Missing propositions}
\begin{proposition}
	Let $\mathcal{E}$ be a matrix formed by $[\epsilon_1,\epsilon_2,...,\epsilon_k]$, where $\epsilon_i$ has mean $\|\nu_i\|\leq  \nu$ and variance $\Sigma_i \preceq \sigma \textbf{I}$. By triangle inequality then Jensen's inequality, we have
	\BEQ
	\E[\|\mathcal{E}\|_2] \leq \sum_{i=0}^k  \E[\|\varepsilon_i\|] \leq \sum_{i=0}^k\sqrt{\E [\|\varepsilon_i\|^2]} \leq O(\nu + \sigma) \label{eq:jensen_ineq_matrix}.
	\EEQ
\end{proposition}

\begin{proposition} \label{prop:opt_val_sqrt_fun}
	Consider the function
	\[
	f(x) = \frac{1}{\kappa}\sqrt{a - \lambda x^2} + bx 
	\]
	defined for $x \in [0,\sqrt{a/\lambda}]$. The its maximal value is attained at
	\[
	x_{ \text{opt} } = \frac{b \sqrt{a}}{\sqrt{\frac{\lambda^2}{\kappa^2} + \lambda b^2}},
	\]
	and its maximal value is thus, if $x_{ \text{opt} } \in [0,\sqrt{a/\lambda}]$,
	\BEA
	f_{\max} =  \sqrt{a} \sqrt{\frac{1}{\kappa^2} + \frac{b^2}{\lambda}} \label{eq:opt_val_sqrt_fun}.
	\EEA
\end{proposition}
\begin{proof}
	The (positive) root of the derivative of $f$ follows
	\[
	b\sqrt{a-\lambda x^2} - \frac{1}{\kappa}\lambda x = 0 \qquad \Leftrightarrow \qquad x = \frac{b \sqrt{a}}{\sqrt{\frac{\lambda^2}{\kappa^2} + \lambda b^2}} .
	\]
	If we inject the solution in our function, we obtain its maximal value,
	\BEAS
	\frac{1}{\kappa} \sqrt{a - \lambda \left(\frac{b \sqrt{a}}{\sqrt{\frac{\lambda^2}{\kappa^2} + \lambda b^2}}\right)^2} + b \frac{b \sqrt{a}}{\sqrt{\frac{\lambda^2}{\kappa^2} + \lambda b^2}} & = & \frac{1}{\kappa} \sqrt{a - \lambda \frac{b^2 a}{\frac{\lambda^2}{\kappa^2} + \lambda b^2}} + b \frac{b \sqrt{a}}{\sqrt{\frac{\lambda^2}{\kappa^2} + \lambda b^2}}, \\
	& = & \frac{1}{\kappa} \sqrt{a - \lambda \frac{b^2 a}{\frac{\lambda^2}{\kappa^2} + \lambda b^2}} + b \frac{b \sqrt{a}}{\sqrt{\frac{\lambda^2}{\kappa^2} + \lambda b^2}}, \\
	& = & \frac{1}{\kappa} \sqrt{\frac{a \lambda^2 \frac{1}{\kappa^2}}{\frac{\lambda^2}{\kappa^2} + \lambda b^2}} + b \frac{b \sqrt{a}}{\sqrt{\frac{\lambda^2}{\kappa^2} + \lambda b^2}}, \\
	& = &  \sqrt{a} \frac{ \frac{1}{\kappa^2} \lambda + b^2 }{\sqrt{\frac{\lambda^2}{\kappa^2} + \lambda b^2}}, \\
	& = &  \frac{\sqrt{a}}{\lambda} \sqrt{\frac{\lambda^2}{\kappa^2} + \lambda b^2}.
	\EEAS
	The simplification with $\lambda$ in the last equality concludes the proof.
\end{proof}

\clearpage
\subsection{Additional numerical experiments}

\subsubsection{Legend}
~\\
\begin{figure}[!htb]
\psfrag{data1}{\scriptsize SAGA}
\psfrag{data2}{\scriptsize Sgd}
\psfrag{data3}{\scriptsize SVRG}
\psfrag{data4}{\scriptsize Katyusha}
\psfrag{data5}{\scriptsize AccSAGA}
\psfrag{data6}{\scriptsize AccSgd}
\psfrag{data7}{\scriptsize AccSVRG}
\psfrag{data8}{\scriptsize AccKat.}
\includegraphics[width=\textwidth]{figs/classif/legend.eps}
\end{figure}

\subsubsection{datasets}
~\\
\begin{table}[!htb]
	\centering
	\begin{tabular}{rcccc}
		& Sonar UCI (\texttt{Son}) & Madelon UCI (\texttt{Mad}) & Random (\texttt{Ran}) &  Sido0 (\texttt{Sid})\\
		\hline
		\# samples $N$ & $208$ & $2000$ & $4000$ & $12678$ \\
		Dimension $d$ & $60$ & $500$ & $1500$ & $4932$
	\end{tabular}
	\vspace{0.2cm}
	\caption{Datasets used in the experiments.}
	\label{table:datasets}
\end{table}

\clearpage
\subsubsection{Quadratic loss}

\paragraph{Sonar dataset}
~
\begin{figure}[h!]
	\centering
	\psfrag{valf}[b]{\scriptsize $f(x)-f(x^*)$}
	\psfrag{xlabel}{\footnotesize Epoch}
	\includegraphics[width=0.45\textwidth]{figs/classif/leastsquare_sonar_well_cond_208_60_nite.eps}
	\psfrag{xlabel}{\footnotesize Time (sec)}
	\psfrag{valf}[b]{\scriptsize }
	\includegraphics[width=0.45\textwidth]{figs/classif/leastsquare_sonar_well_cond_208_60_time.eps}

	\psfrag{valf}[b]{\scriptsize $f(x)-f(x^*)$}
	\psfrag{xlabel}{\footnotesize Epoch}
	\includegraphics[width=0.45\textwidth]{figs/classif/leastsquare_sonar_regular_cond_208_60_nite.eps}
	\psfrag{xlabel}{\footnotesize Time (sec)}
	\psfrag{valf}[b]{\scriptsize }
	\includegraphics[width=0.45\textwidth]{figs/classif/leastsquare_sonar_regular_cond_208_60_time.eps}

	\psfrag{valf}[b]{\scriptsize $f(x)-f(x^*)$}
	\psfrag{xlabel}{\footnotesize Epoch}
	\includegraphics[width=0.45\textwidth]{figs/classif/leastsquare_sonar_bad_cond_208_60_nite.eps}
	\psfrag{xlabel}{\footnotesize Time (sec)}
	\psfrag{valf}[b]{\scriptsize }
	\includegraphics[width=0.45\textwidth]{figs/classif/leastsquare_sonar_bad_cond_208_60_time.eps}
	
	\caption{Quadratic loss  with (top to bottom) good, moderate and bad conditioning using \texttt{Son} dataset.}
\end{figure}

\clearpage
\paragraph{Madelon dataset}
~
\begin{figure}[h!]
	\centering
	\psfrag{valf}[b]{\scriptsize $f(x)-f(x^*)$}
	\psfrag{xlabel}{\footnotesize Epoch}
	\includegraphics[width=0.45\textwidth]{figs/classif/leastsquare_madelon_well_cond_2000_500_nite.eps}
	\psfrag{xlabel}{\footnotesize Time (sec)}
	\psfrag{valf}[b]{\scriptsize }
	\includegraphics[width=0.45\textwidth]{figs/classif/leastsquare_madelon_well_cond_2000_500_time.eps}

	\psfrag{valf}[b]{\scriptsize $f(x)-f(x^*)$}
	\psfrag{xlabel}{\footnotesize Epoch}
	\includegraphics[width=0.45\textwidth]{figs/classif/leastsquare_madelon_regular_cond_2000_500_nite.eps}
	\psfrag{xlabel}{\footnotesize Time (sec)}
	\psfrag{valf}[b]{\scriptsize }
	\includegraphics[width=0.45\textwidth]{figs/classif/leastsquare_madelon_regular_cond_2000_500_time.eps}

	\psfrag{valf}[b]{\scriptsize $f(x)-f(x^*)$}
	\psfrag{xlabel}{\footnotesize Epoch}
	\includegraphics[width=0.45\textwidth]{figs/classif/leastsquare_madelon_bad_cond_2000_500_nite.eps}
	\psfrag{xlabel}{\footnotesize Time (sec)}
	\psfrag{valf}[b]{\scriptsize }
	\includegraphics[width=0.45\textwidth]{figs/classif/leastsquare_madelon_bad_cond_2000_500_time.eps}
	
	\caption{Quadratic loss  with (top to bottom) good, moderate and bad conditioning using \texttt{Mad} dataset.}
\end{figure}

\clearpage
\paragraph{Random dataset}
~
\begin{figure}[h!]
	\centering
	\psfrag{valf}[b]{\scriptsize $f(x)-f(x^*)$}
	\psfrag{xlabel}{\footnotesize Epoch}
	\includegraphics[width=0.45\textwidth]{figs/classif/leastsquare_random_well_cond_4000_1500_nite.eps}
	\psfrag{xlabel}{\footnotesize Time (sec)}
	\psfrag{valf}[b]{\scriptsize }
	\includegraphics[width=0.45\textwidth]{figs/classif/leastsquare_random_well_cond_4000_1500_time.eps}

	\psfrag{valf}[b]{\scriptsize $f(x)-f(x^*)$}
	\psfrag{xlabel}{\footnotesize Epoch}
	\includegraphics[width=0.45\textwidth]{figs/classif/leastsquare_random_regular_cond_4000_1500_nite.eps}
	\psfrag{xlabel}{\footnotesize Time (sec)}
	\psfrag{valf}[b]{\scriptsize }
	\includegraphics[width=0.45\textwidth]{figs/classif/leastsquare_random_regular_cond_4000_1500_time.eps}

	\psfrag{valf}[b]{\scriptsize $f(x)-f(x^*)$}
	\psfrag{xlabel}{\footnotesize Epoch}
	\includegraphics[width=0.45\textwidth]{figs/classif/leastsquare_random_bad_cond_4000_1500_nite.eps}
	\psfrag{xlabel}{\footnotesize Time (sec)}
	\psfrag{valf}[b]{\scriptsize }
	\includegraphics[width=0.45\textwidth]{figs/classif/leastsquare_random_bad_cond_4000_1500_time.eps}
	
	\caption{Quadratic loss  with (top to bottom) good, moderate and bad conditioning using \texttt{Ran} dataset.}
\end{figure}

\clearpage
\paragraph{Sido0 dataset}
~

\begin{figure}[h!]
	\centering
	\psfrag{valf}[b]{\scriptsize $f(x)-f(x^*)$}
	\psfrag{xlabel}{\footnotesize Epoch}
	\includegraphics[width=0.45\textwidth]{figs/classif/leastsquare_sido0_well_cond_12678_4932_nite.eps}
	\psfrag{xlabel}{\footnotesize Time (sec)}
	\psfrag{valf}[b]{\scriptsize }
	\includegraphics[width=0.45\textwidth]{figs/classif/leastsquare_sido0_well_cond_12678_4932_time.eps}

	\psfrag{valf}[b]{\scriptsize $f(x)-f(x^*)$}
	\psfrag{xlabel}{\footnotesize Epoch}
	\includegraphics[width=0.45\textwidth]{figs/classif/leastsquare_sido0_regular_cond_12678_4932_nite.eps}
	\psfrag{xlabel}{\footnotesize Time (sec)}
	\psfrag{valf}[b]{\scriptsize }
	\includegraphics[width=0.45\textwidth]{figs/classif/leastsquare_sido0_regular_cond_12678_4932_time.eps}

	\psfrag{valf}[b]{\scriptsize $f(x)-f(x^*)$}
	\psfrag{xlabel}{\footnotesize Epoch}
	\includegraphics[width=0.45\textwidth]{figs/classif/leastsquare_sido0_bad_cond_12678_4932_nite.eps}
	\psfrag{xlabel}{\footnotesize Time (sec)}
	\psfrag{valf}[b]{\scriptsize }
	\includegraphics[width=0.45\textwidth]{figs/classif/leastsquare_sido0_bad_cond_12678_4932_time.eps}
	
	\caption{Quadratic loss  with (top to bottom) good, moderate and bad conditioning using \texttt{Sid} dataset.}
\end{figure}

\clearpage
\subsubsection{Logistic loss}

\paragraph{Sonar dataset}
~
\begin{figure}[h!]
	\centering
	\psfrag{valf}[b]{\scriptsize $f(x)-f(x^*)$}
	\psfrag{xlabel}{\footnotesize Epoch}
	\includegraphics[width=0.45\textwidth]{figs/classif/logistic_sonar_well_cond_208_60_nite.eps}
	\psfrag{xlabel}{\footnotesize Time (sec)}
	\psfrag{valf}[b]{\scriptsize }
	\includegraphics[width=0.45\textwidth]{figs/classif/logistic_sonar_well_cond_208_60_time.eps}

	\psfrag{valf}[b]{\scriptsize $f(x)-f(x^*)$}
	\psfrag{xlabel}{\footnotesize Epoch}
	\includegraphics[width=0.45\textwidth]{figs/classif/logistic_sonar_regular_cond_208_60_nite.eps}
	\psfrag{xlabel}{\footnotesize Time (sec)}
	\psfrag{valf}[b]{\scriptsize }
	\includegraphics[width=0.45\textwidth]{figs/classif/logistic_sonar_regular_cond_208_60_time.eps}

	\psfrag{valf}[b]{\scriptsize $f(x)-f(x^*)$}
	\psfrag{xlabel}{\footnotesize Epoch}
	\includegraphics[width=0.45\textwidth]{figs/classif/logistic_sonar_bad_cond_208_60_nite.eps}
	\psfrag{xlabel}{\footnotesize Time (sec)}
	\psfrag{valf}[b]{\scriptsize }
	\includegraphics[width=0.45\textwidth]{figs/classif/logistic_sonar_bad_cond_208_60_time.eps}
	
	\caption{Logistic loss  with (top to bottom) good, moderate and bad conditioning using \texttt{Son} dataset.}
\end{figure}

\clearpage
\paragraph{Madelon dataset}
~
\begin{figure}[h!]
	\centering
	\psfrag{valf}[b]{\scriptsize $f(x)-f(x^*)$}
	\psfrag{xlabel}{\footnotesize Epoch}
	\includegraphics[width=0.45\textwidth]{figs/classif/logistic_madelon_well_cond_2000_500_nite.eps}
	\psfrag{xlabel}{\footnotesize Time (sec)}
	\psfrag{valf}[b]{\scriptsize }
	\includegraphics[width=0.45\textwidth]{figs/classif/logistic_madelon_well_cond_2000_500_time.eps}

	\psfrag{valf}[b]{\scriptsize $f(x)-f(x^*)$}
	\psfrag{xlabel}{\footnotesize Epoch}
	\includegraphics[width=0.45\textwidth]{figs/classif/logistic_madelon_regular_cond_2000_500_nite.eps}
	\psfrag{xlabel}{\footnotesize Time (sec)}
	\psfrag{valf}[b]{\scriptsize }
	\includegraphics[width=0.45\textwidth]{figs/classif/logistic_madelon_regular_cond_2000_500_time.eps}

	\psfrag{valf}[b]{\scriptsize $f(x)-f(x^*)$}
	\psfrag{xlabel}{\footnotesize Epoch}
	\includegraphics[width=0.45\textwidth]{figs/classif/logistic_madelon_bad_cond_2000_500_nite.eps}
	\psfrag{xlabel}{\footnotesize Time (sec)}
	\psfrag{valf}[b]{\scriptsize }
	\includegraphics[width=0.45\textwidth]{figs/classif/logistic_madelon_bad_cond_2000_500_time.eps}
	
	\caption{Logistic loss  with (top to bottom) good, moderate and bad conditioning using \texttt{Mad} dataset.}
\end{figure}

\clearpage
\paragraph{Random dataset}
~
\begin{figure}[h!]
	\centering
	\psfrag{valf}[b]{\scriptsize $f(x)-f(x^*)$}
	\psfrag{xlabel}{\footnotesize Epoch}
	\includegraphics[width=0.45\textwidth]{figs/classif/logistic_random_well_cond_4000_1500_nite.eps}
	\psfrag{xlabel}{\footnotesize Time (sec)}
	\psfrag{valf}[b]{\scriptsize }
	\includegraphics[width=0.45\textwidth]{figs/classif/logistic_random_well_cond_4000_1500_time.eps}

	\psfrag{valf}[b]{\scriptsize $f(x)-f(x^*)$}
	\psfrag{xlabel}{\footnotesize Epoch}
	\includegraphics[width=0.45\textwidth]{figs/classif/logistic_random_regular_cond_4000_1500_nite.eps}
	\psfrag{xlabel}{\footnotesize Time (sec)}
	\psfrag{valf}[b]{\scriptsize }
	\includegraphics[width=0.45\textwidth]{figs/classif/logistic_random_regular_cond_4000_1500_time.eps}

	\psfrag{valf}[b]{\scriptsize $f(x)-f(x^*)$}
	\psfrag{xlabel}{\footnotesize Epoch}
	\includegraphics[width=0.45\textwidth]{figs/classif/logistic_random_bad_cond_4000_1500_nite.eps}
	\psfrag{xlabel}{\footnotesize Time (sec)}
	\psfrag{valf}[b]{\scriptsize }
	\includegraphics[width=0.45\textwidth]{figs/classif/logistic_random_bad_cond_4000_1500_time.eps}
	
	\caption{Logistic loss  with (top to bottom) good, moderate and bad conditioning using \texttt{Ran} dataset.}
\end{figure}

\clearpage
\paragraph{Sido0 dataset}
~

\begin{figure}[h!]
	\centering
	\psfrag{valf}[b]{\scriptsize $f(x)-f(x^*)$}
	\psfrag{xlabel}{\footnotesize Epoch}
	\includegraphics[width=0.45\textwidth]{figs/classif/logistic_sido0_well_cond_12678_4932_nite.eps}
	\psfrag{xlabel}{\footnotesize Time (sec)}
	\psfrag{valf}[b]{\scriptsize }
	\includegraphics[width=0.45\textwidth]{figs/classif/logistic_sido0_well_cond_12678_4932_time.eps}

	\psfrag{valf}[b]{\scriptsize $f(x)-f(x^*)$}
	\psfrag{xlabel}{\footnotesize Epoch}
	\includegraphics[width=0.45\textwidth]{figs/classif/logistic_sido0_regular_cond_12678_4932_nite.eps}
	\psfrag{xlabel}{\footnotesize Time (sec)}
	\psfrag{valf}[b]{\scriptsize }
	\includegraphics[width=0.45\textwidth]{figs/classif/logistic_sido0_regular_cond_12678_4932_time.eps}

	\psfrag{valf}[b]{\scriptsize $f(x)-f(x^*)$}
	\psfrag{xlabel}{\footnotesize Epoch}
	\includegraphics[width=0.45\textwidth]{figs/classif/logistic_sido0_bad_cond_12678_4932_nite.eps}
	\psfrag{xlabel}{\footnotesize Time (sec)}
	\psfrag{valf}[b]{\scriptsize }
	\includegraphics[width=0.45\textwidth]{figs/classif/logistic_sido0_bad_cond_12678_4932_time.eps}
	
	\caption{Logistic loss  with (top to bottom) good, moderate and bad conditioning using \texttt{Sid} dataset.}
\end{figure}

\clearpage

{\small 
	\bibliographystyle{agsm}
	\bibliography{reg_non_acc}
}

\end{document}